\documentclass[11pt]{amsart}

\def\NZQ{\mathbb}               
\def\NN{{\NZQ N}}
\def\QQ{{\NZQ Q}}
\def\ZZ{{\NZQ Z}}
\def\RR{{\NZQ R}}

\def\FF{{\NZQ F}}

\def\frk{\mathfrak}               

\def\mm{{\frk m}}

\def\nn{{\frk n}}
\def\Phi{{\frk N}}
\def\opn#1#2{\def#1{\operatorname{#2}}} 
\opn\chara{char} \opn\length{\ell} \opn\pd{pd} \opn\rk{rk}
\opn\projdim{proj\,dim} \opn\injdim{inj\,dim} \opn\rank{rank}
\opn\depth{depth} \opn\grade{grade} \opn\height{height}
\opn\embdim{emb\,dim} \opn\codim{codim}

\opn\Tr{Tr} \opn\bigrank{big\,rank}
\opn\superheight{superheight}\opn\lcm{lcm}
\opn\trdeg{tr\,deg}
\opn\reg{reg} \opn\lreg{lreg} \opn\ini{in} \opn\lpd{lpd}
\opn\size{size}\opn{\mult}{mult}
\opn\div{div} \opn\Div{Div} \opn\cl{cl} \opn\Cl{Cl}
\opn\Spec{Spec} \opn\Supp{Supp} \opn\supp{supp} \opn\Sing{Sing}
\opn\Ass{Ass} \opn\Min{Min}
\opn\Ann{Ann} \opn\Rad{Rad} \opn\Soc{Soc}
\opn\Syz{Syz} \opn\Im{Im} \opn\Ker{Ker} \opn\Coker{Coker}
\opn\Am{Am} \opn\Hom{Hom} \opn\Tor{Tor} \opn\Ext{Ext}
\opn\End{End} \opn\Aut{Aut} \opn\id{id}

\opn\nat{nat}
\opn\pff{pf}
\opn\Pf{Pf} \opn\GL{GL} \opn\SL{SL} \opn\mod{mod} \opn\ord{ord}
\opn\Gin{Gin}
\opn\Hilb{Hilb}\opn\adeg{adeg}\opn\std{std}\opn\ip{infpt}
\opn\Pol{Pol}
\opn\sat{sat}
\opn\Var{Var}
\opn\Gen{Gen}

\opn\aff{aff} \opn\con{conv} \opn\relint{relint} \opn\st{st}
\opn\lk{lk} \opn\cn{cn} \opn\core{core} \opn\vol{vol}
\opn\link{link} \opn\star{star}
\opn\gr{gr}

\def\pot#1#2{#1[\kern-0.28ex[#2]\kern-0.28ex]}

\opn\dirlim{\underrightarrow{\lim}}
\opn\inivlim{\underleftarrow{\lim}}
\let\union=\cup
\let\sect=\cap
\let\dirsum=\oplus

\let\iso=\cong
\let\Union=\bigcup

\let\Dirsum=\bigoplus

\let\to=\rightarrow
\let\To=\longrightarrow
\def\Implies{\ifmmode\Longrightarrow \else
        \unskip${}\Longrightarrow{}$\ignorespaces\fi}
\def\implies{\ifmmode\Rightarrow \else
        \unskip${}\Rightarrow{}$\ignorespaces\fi}
\def\iff{\ifmmode\Longleftrightarrow \else
        \unskip${}\Longleftrightarrow{}$\ignorespaces\fi}

\let\:=\colon
\newtheorem{Theorem}{Theorem}[section]
\newtheorem{Lemma}[Theorem]{Lemma}
\newtheorem{Corollary}[Theorem]{Corollary}
\newtheorem{Proposition}[Theorem]{Proposition}

\newtheorem{Example}[Theorem]{Example}

\let\epsilon\varepsilon
\let\phi=\varphi
\let\kappa=\varkappa
\def\qed{\ifhmode\textqed\fi
      \ifmmode\ifinner\quad\qedsymbol\else\dispqed\fi\fi}
\def\textqed{\unskip\nobreak\penalty50
       \hskip2em\hbox{}\nobreak\hfil\qedsymbol
       \parfillskip=0pt \finalhyphendemerits=0}
\def\dispqed{\rlap{\qquad\qedsymbol}}

\opn\dis{dis}
\def\pnt{{\raise0.5mm\hbox{\large\bf.}}}

\opn\Lex{Lex}

\begin{document}

\title{Finite generation of algebras associated to powers of ideals}

\author{Steven Dale Cutkosky, J\"urgen Herzog  and Hema Srinivasan}
\thanks{The first author was partially supported by NSF}

\address{Steven Dale Cutkosky, Mathematics Department,
202 Mathematical Sciences Bldg,
University of Missouri,
Columbia, MO 65211 USA
}\email{dale@math.missouri.edu}

\address{J\"urgen Herzog, Fachbereich Mathematik, Universit\"at Duisburg-Essen, Campus Essen, 45117
Essen, Germany} \email{juergen.herzog@uni-essen.de}

\address{Hema Srinivasan, Mathematics Department,
202 Mathematical Sciences Bldg,
University of Missouri,
Columbia, MO 65211 USA
}\email{srinivasan@math.missouri.edu <srinivasan@math.missouri.edu>}

\begin{abstract}
We study generalized symbolic powers and form ideals of powers and compare their growth with the growth of  ordinary powers, and we discuss the question when the graded rings attached to symbolic powers or to form ideals of powers are finitely generated.
\end{abstract}

\maketitle
\section*{Introduction}

Our starting motivation for this paper is a result of Hoang and Trung in \cite{HT} where they showed that the Hilbert coefficients of the powers $I^k$ of a graded ideal $I$ in the polynomial ring $S=K[x_1,\ldots, x_n]$ are polynomial functions in $k$ for $k\gg 0$. In an explicit form this statement is given in \cite{HPV}. In the same  paper the question is raised whether for any ideal  in a Noetherian local ring $(R,\mm, K)$ a similar statement is true. We do not have any counterexample yet.  On the other hand,  a positive answer is unlikely by the following reason: denote by $G(R/I)$ the associated graded  ring of $R/I$ with respect to the maximal ideal $\mm$. Note  that $G(R/I)=G(R)/I^*$ where $I^*\subset G(R)$ is the graded ideal in $G(R)$ generated by all elements $f^*$ with $f\in I$, where $f^*$ is the leading form of $f$, defined as follows: let $d=\sup\{j\: f\in \mm^j\}$, then $f^*=f+\mm^{d+1}$. With this notation introduced we see that Hilbert function $H(R/I^k,j)$ of $R/I$ is given by
$H(R/I^k,j) =\sum_{i=0}^j\dim_K(G(R)/(I^k)^*)_i$.  In the graded case $(I^k)^*=I^k$, so that in this case  the algebra $A(I)=\Dirsum_{k\geq 0}(I^k)^*$ is finitely generated over $R$, indeed is equal to the Rees algebra of $I$ and hence is standard graded over $R$. This fact is substantially used in the proof of Hoang and Trung. Unfortunately, in general $A(I)$ is  not finitely generated, even if $I$ is generated by quasi-homogeneous polynomials.  We give such an example in \ref{infinite}. There we show that for the ideal $I=(x^2, y^3-xy)\subset K[[x,y]]$ the algebra $A(I)$ is not finitely generated. Notice that $I$ is quasi-homogeneous if we set $\deg x=2$ and $\deg y=1$. This example is also remarkable, since it is an $\mm$-primary  complete intersection.  We also show in Example~\ref{depending} that the finite generation of $A(I)$ may depend on the characteristic of the residue class field.
General criteria for the finite generation of $A(I)$ seem to be not available. However if $I$ is an ideal in the power series ring $R=K[[x_1,\ldots,x_n]]$ the following strategy can be applied: as explained in Lemma~\ref{phi} and Proposition~\ref{comparison} there is attached to $I$ in natural way an ideal $I^{\sharp}\subset R[[s]]$ with the property that $A(I)$ is finitely generated if and only if $\Dirsum_{k\geq 0}(I^{\sharp})^k\: s^\infty$ is finitely generated.

In Section 1 we study more generally algebras of the form $S_J(I)=\Dirsum_{k\geq 0}^\infty I^k\: J^\infty$ where $I$ and $J$ are ideals in a Noetherian local ring or graded ideals in a standard graded $K$-algebra. This type of algebras have been intensively studied in algebraic as well as in combinatorial contexts. If we choose $J=\mm$, then $I^k\: \mm^\infty$ is equal to the saturated power $\widetilde{I^k}$ and $S_\mm(I)$ is the saturated power algebra. In general the saturated power algebra is not finitely generated. Finite generation of $S_\mm(I)$ implies that the regularity of the saturated powers $\widetilde{I^k}$ of $I$ are quasi-linear functions of $k$ for large $k$, see \cite{CHT} and \cite{Ko}.  However  it is shown by examples in \cite{CHT}, \cite{C} and \cite{CEL}  that the  regularity of saturated powers may behave   extremely strangely. In particular in these examples $S_\mm(I)$ cannot be finitely generated.
 Another special case of interest is obtained when we choose for $J$ the intersection of all asymptotic prime ideals of $I$ which are not minimal. In this  case $S_J(I)$ is the symbolic Rees algebra.  This  is why we call  for any choice of $J$ the ideals $I^k\: J^\infty$  generalized symbolic  powers of $I$. If $I$ happens to be a prime ideal, then these are the classical symbolic powers. Finite generation of symbolic Rees algebras has been studied in many papers, for instance in \cite{K}, \cite{Sch}, \cite{KR},  and remarkable examples have been found where these algebras are not finitely generated, see \cite{R} and \cite{GNW}.
 Symbolic Rees algebras of squarefree monomial ideals can be identified with vertex cover algebras. This class of algebras is always finitely generated, as shown in \cite{HHT}.

In Section 2 we address the question under which conditions the algebra $A(I)$ is finitely generated. At the moment we can only offer very partial results. For example, we show in Corollary~\ref{application} that if
 $(R,\mm)$ is a regular local ring and  $I\subset R$  a complete intersection ideal with $\dim R/I=0$, and if either $R$  is 2-dimensional or $I^*$ is a monomial ideal, then the following conditions are equivalent: (a) $A(I)$ is standard graded, (b) $I^*$ is a complete intersection, (c)  for infinitely many integers  $k$ we have  $(I^k)^*=(I^*)^k$. It would be interesting to know whether these equivalent conditions hold without the extra assumptions on $R$ or $I^*$.

 In this paper we give a criterion for finite generation of $S_J(I)$.  We show in Theorem~\ref{strengthen} that if $R$ is an excellent local domain,  $I$ and $J$ are proper ideals of $R$ and  $\depth R_P\geq 2$ for all $P\in V(J)$, then $S_J(I)$ is finitely generated if and only if there exists an integer $r>0$ such that $\ell((I^r\: J^\infty)_P)<R_P$ for all $P\in V(J)$. Here $\ell(H)$ denotes the analytic spread of an ideal $H$. For ordinary symbolic powers a related result   was proven by Katz and Ratliff \cite[Theorem A and Corollary 1]{KR}. One direction of Theorem~\ref{strengthen} follows from part (a) of  Theorem~\ref{small}, which is inspired by a result of McAdam \cite{M}, where we give a short  direct proof of the fact that under the above conditions on $R$ and $J$, $S_J(I)$ is a graded subalgebra of the integral closure of the Rees algebra of $I$,   provided that $\ell(I_P)<\dim R_P$ for all $P\in V(J)$. This result can also  be deduced from Theorem~4.1 of Katz's paper [K] and Theorem~5.6 of Schenzel's paper [S].  In the second part of  Theorem~\ref{small} we also show that  $\lim_{k\to \infty}e((I^k:J^\infty)/I^k)/k^{\ell(I)+\dim R/J-1}$ exists and is a rational number. Here $e(M)$ denotes the multiplicity of a module $M$. In particular it follows from the above results that the
 saturated power algebra $\Dirsum_{k\geq 0}\widetilde{I^k}$ is finitely generated if $\ell(\widetilde{I})<\dim R$, and that in this case $\lim_{k\to \infty}\lambda(\widetilde{I^k}/I^k)/k^{\ell(I)-1}$ exists and is a rational number.

 It is quite interesting that the saturated powers of an ideal with $\ell(I)=\dim R$ behave quite differently. In fact,  as a complement to a result in \cite{CHST} given there for graded ideals, but in this paper with a restriction on the ring and the singular locus of $\mbox{Spec}(R/I)$, we show  in Theorem~\ref{nice}  the following: suppose that $(R,\mm)$ is a regular local ring of dimension $d$, which is essentially of finite type over a  field $K$ of characteristic zero. Suppose that $I\subset R$ is an ideal such that
the singular locus of $\mbox{Spec}(R/I)$ is $\{\mm\}$.  Then the limit
$
\lim_{k\rightarrow\infty} \lambda(\widetilde{I^k}/I^k)/k^d\in\RR
$
exists. This limit may be indeed an irrational number as shown in \cite{CHST}. Comparing this result with the above statements, we see that $\lim_{k\rightarrow\infty} \lambda(\widetilde{I^k}/I^k)/k^d=0$ if $\ell(I)<\dim R$, and that $\lim_{k\rightarrow\infty} \lambda(\widetilde{I^k}/I^k)/k^{d-1}$ never exists. It would be interesting to know whether we always have that $\lim_{k\rightarrow\infty} \lambda(\widetilde{I^k}/I^k)/k^{\ell(I)}\neq 0$ if $\ell(I)=\dim R$, and $\lim_{k\rightarrow\infty} \lambda(\widetilde{I^k}/I^k)/k^{\ell(I)-1}\neq 0$ if $\ell(I)<\dim R$. In the case where
$\ell(I) = d$ we do not have a counterexample. However, when $\ell(I) < d$ we do have
counterexamples to this statement. For example, consider $I = (x_1, \cdots , x_r) \subset K[[x_1, \ldots,  x_n]]$ with $r < n$. In this case $\ell(I) = r < n$ and $\widetilde{I^k} = I^k$ for all $k$. So the limit is certainly zero. We  do not know examples of other types of growth (such as $n^i$ with $0 < i < \ell(I) - 1$).

The authors want to thank Bernd Ulrich for several useful discussion concerning Theorem~\ref{complete}.

\section{Generalized symbolic powers}
Let $(R,\mm)$ be a local ring or a positively graded $K$-algebra with graded maximal ideal $\mm$, where $K$ is a field,  and let $I$ and $J$ be proper ideals in $R$ which are graded if $R$ is graded. In this section we want to study the algebra $S_J(I)=\Dirsum_{k\geq 0}I^k\: J^\infty$ of generalized symbolic powers of $I$ with respect to $J$. The Rees ring of $I$ will be denoted by $R(I)$ and its integral closure by $\overline{R(I)}$ in case $R$ is a domain. It turns out that the analytic spread $\ell(I)$ of $I$, which is defined to be the Krull dimension of $R(I)/\mm R(I)$, plays an important role  in the study of these algebras. The multiplicity of a finitely generated $R$-module $M$ will be denoted by $e(M)$.

Theorem~\ref{small} below is inspired by results of MacAdam \cite{M}, Ratliff \cite{R1},
Katz \cite{K}, Schenzel \cite{Sch}  and others on the
asymptotic associated primes of ideals of small analytic spread. Part (a) of Theorem~\ref{small} follows from Theorem 4.1 of Katz's paper \cite{K} and Theorem 5.6 of Schenzel's paper \cite{Sch}. We give a self contained proof for the reader's convenience.

\begin{Theorem}
\label{small}
Let $(R,\mm)$ be an excellent  domain.   Assume that for all $P\in V(J)$ we have that  {\em (i)} $\depth R_P\geq 2$, and {\em (ii)} $\ell(I_P)< \dim R_P$. Then
\begin{enumerate}
\item[{\em (a)}] $S_J(I)\subset \overline{R(I)}$. In particular, $S_J(I)$ is a finitely generated $R$-algebra.

\item[{\em (b)}] $\lim_{k\to \infty}e((I^k:J^\infty)/I^k)/k^{\ell(I)+\dim R/J-1}$ exists and is a rational number.
\end{enumerate}
\end{Theorem}

\begin{proof}
(a) Since $I^k\: J^\infty\subset \overline{I^k}\: J^\infty$ for all $k$, it suffices to show that $\overline{I^k}\: J^\infty=\overline{I^k}$ for all $k$.
Recall that (see \cite[Proposition 1.2.10]{BH})
\begin{eqnarray}
\label{grade}
\grade(J,R)=\inf\{\depth R_P,\; P\in V(J)\}.
\end{eqnarray}
Thus assumption (i) implies that $\grade(J,R)\geq 2$, so that $H^0_J(R)=H^1_J(R)=0$, see \cite[Theorem 6.2.7]{BS}. It follows that $\overline{I^k}\: J^\infty/\overline{I^k}=H^0_J(R/\overline{I^k})\iso H^1_J(\overline{I^k})$ for all $k$. Therefore, if $A=\overline{R(I)}$ denotes the integral closure of $R(I)$,  it remains to be shown that
$H^1_J(A)=0$  which, by \cite[Theorem 6.2.7]{BS},  is equivalent to saying that $\grade(JA,A)\geq 2$.

We apply again (\ref{grade}), this time to the ideal $JA$ and the ring $A$, and obtain that
$$\grade(JA,A)=\inf\{\depth A_Q,\; Q\in V(JA)\}.$$ Since $A$ is normal it satisfies Serre's condition $S_2$. In other words, $\depth A_Q\geq 2$ for all $Q\in \Spec(A)$ with $\dim A_Q\geq 2$. Thus we need to show that $\dim A_Q\geq 2$ for all $Q\in V(JA)$. Let $P=Q\sect R$. Then $P\in V(J)$. Localizing at $P$ we may assume that $P=\mm$. Our assumption on $R$ guarantees that
$\dim A_Q=\dim A-\dim A/Q$. Notice that $\dim A=\dim R(I)=\dim R+1$, since $A$ is a finite $R(I)$-module, and that $\dim A/Q\leq \dim A/\mm A=\dim R(I)/\mm R(I)=\ell(I)$. Thus condition (ii) implies that $\dim A_Q\geq \dim R+1-\ell(I)\geq 2$, as desired.

(b) By part (a), the algebra $S_J(I)$ is finitely generated. Thus  \cite[Theorem 3.2]{HPV} implies that there exist  polynomials $P_0, \ldots, P_{g-1}$, all of same degree and with same leading coefficient,  such that  $e(I^{mg+i}:J^\infty/J^{mg+i})=P_i(m)$  for $i=0,\ldots, g-1$ and all $m\gg 0$. For each $i$, the modules $I^{mg+i}:J^\infty/J^{mg+i}$ have constant dimension for $m\gg 0$, say $d_i$. Since they are supported in $V(J)$ it follows that $d_i\leq \dim R/J$ for all $i$. Thus applying \cite[Proposition 5.5]{HPV} we see that $\deg P_i\leq \ell(I)+\dim R/J-1$ for all $i$, and thus statement (b) follows.
\end{proof}

\begin{Example}
\label{marc}
{\em (Marc Chardin) Let $I=(xw-yz,x^2,z^2)\subset S=K[x,y,z,w]$.  Then $\ell(I)=3<\dim S$ and $e(\widetilde{I^k}/I^k)=\lambda(\widetilde{I^k}/I^k)={k+1\choose 2}$. It follows   that
$$\lim_{k\to \infty} e(\widetilde{I^k}/I^k)/k^{\ell(I)-1}=1/2.$$
This example shows that the limit would not exist, if we would choose smaller power of $k$ than $\ell(I)-1$.

To prove this we consider a presentation of the Rees ring $R=S[y_1,y_2,y_3]\to R(I)$ with $y_1\mapsto (xw-yz)t$, $y_2\mapsto x^2t$ and $y_3\mapsto z^2t$. The kernel $J$ of this map is:
\begin{eqnarray*}
(z^2y_1 + yzy_3 - xwy_3,& yzy_1 + xwy_1 - w^2y_2 + y^2y_3, \\ xzy_1 - zwy_2 + xyy_3, & z^2y_2 - x^2y_3, -x^2y_1 - yzy_2 + xwy_2).
\end{eqnarray*}
Thus $\ell(I)=3$.

If we set $\deg x= \deg y =\deg y=\deg z=(1,0)$ and $\deg y_i=(0,1)$ then $R(I)$ is a bigraded algebra with
\[
R(I)_{(j,k)}=(I^k)_{j+2k} \quad \text{for all $j$ and $k$.}
\]
The bigraded $R$-resolution of $R(I)$ is of the form
\[
\FF: 0\to R(-4,-1)\to R(-3,-1)^4\dirsum R(-2,-2)\to R(-2,-1)^5\to R\to R(I)\to 0
\]
where the last map $R(-4,-1)\to R(-3,-1)^4\dirsum R(-2,-2)$ is given by the $1\times 5$-matrix  $(-x,y,-z,w,0)^t$.

If we take the graded pieces $\FF_{(*,k)}=\Dirsum_j \FF_{(j,k)}$ of the resolution we obtain the exact sequences of $S$-modules
\begin{eqnarray*}
0\to R(-4,-1)_{(*,j)}\to R(-3,-1)_{(*,j)}^4\dirsum & R(-2,-2)_{(*,k)}&\to \\
& R(-2,-1)_{(*,k)}^5& \to R_{(*,k)}\to I^k(2k)\to 0.
\end{eqnarray*}
Since $R(-a,-b)_{(*,k)}=\Dirsum_{a_1+a_2+a_3=k-b} S(-a)y_1^{a_1}y_2^{a_2}y_3^{a_3}\iso S(-a-2k)^{k-b+2\choose 2}$, we obtain for each $k$ the free $S$-resolution
\begin{eqnarray*}
0\to S(-4)^{{k+1\choose 2}}\to S(-3)^{4{k+1 \choose 2}}\dirsum & S(-2)^{{k\choose 2}}& \to \\
& S(-2)^{5{k+1\choose 2}}&\to S^{k+2\choose 2}\to I^{k}(2k)\to 0,
\end{eqnarray*}
where the last map in the resolution maps the basis element $y_1^{a_1}y_2^{a_2}y_3^{a_3}$ with $a_1+a_2+a_3=k-1$ to
$-x y_1^{a_1}y_2^{a_2}y_3^{a_3}-y y_1^{a_1}y_2^{a_2}y_3^{a_3}-z y_1^{a_1}y_2^{a_2}y_3^{a_3}+w y_1^{a_1}y_2^{a_2}y_3^{a_3}.$
Since the cokernel of  the transpose of this map  is $\Ext^3_S(I^k,S)$ we obtain, applying local duality,
\begin{eqnarray}
\label{second}
\widetilde{I^k}/I^k\iso H^0_\mm(S/I^k)\iso \Ext^4_S(S/I^k,S)^\vee\iso \Ext^3_S(I^k,S)^\vee\iso K^{{k+1\choose 2}},
\end{eqnarray}
as desired. Here $N^\vee$ denotes the dual of $N$ with respect to the injective hull of $K$.}
\end{Example}

Let  $R$ be the polynomial ring in $n$ variables over a field of characteristic $0$ and $I\subset R$ be a graded ideal. In \cite[Theorem 0.1]{CHST} it is shown that $\lim_{k\to \infty}\length(\widetilde{I^k}/I^k)/k^n$ exists, but may be an irrational number. Of course, according to Theorem~\ref{small}, this limit can be an irrational number only if $\ell(I)=\dim R$.  In the  following example this limit is a nonzero rational number.

\begin{Example} {\em Let $I=(xy,xz,yz)$. It is easily seen that $\ell(I)=\dim K[x,y,z]=3$ and that $\widetilde{I^2}=(I^2,xyz)$. It is shown in \cite[Proposition 5.3 and Example 4.7]{HHT}  that
\[
\widetilde{I^{2k}}=(\widetilde{I^2})^k= (I^2,xyz)^k= (I^{2k},xyzI^{2(k-1)},\ldots, (xyz)^jI^{2(k-j)},\ldots, (xyz)^k) \quad \text{for all}\quad k.
\]
Let  $\Gen(I^k)$ denotes the minimal set of monomial generators of $I^k$. We claim that set of the monomials $B=\Union_{j=1}^{k}(xyz)^jG(I^{2(k-j)})$  forms $K$-basis of $\widetilde{I^{2k}}/I^{2k}$. Indeed, for all $j=0,\ldots,k$ we have
\begin{eqnarray*}
\{x,y,z\}(xyz)^j\Gen(I^{2(k-j)})&=& (xyz)^{j-1}\{x,y,z\}(xyz)\Gen(I^{2(k-j)})\\
&\subset & (xyz)^{j-1}\Gen(I^{2k})\Gen(I^{2(k-j)})
=(xyz)^{j-1}\Gen(I^{2(k-j+1)}).
\end{eqnarray*}
This together with (\ref{second}) implies that $B$ is system of generators of the $K$-vector space $\widetilde{I^{2k}}/I^{2k}$.
Since for  $j=1,\ldots, k$ the degree of the elements  in $xyz^j\Gen(I^{2(k-j)})$ is $4k-j$, it follows that the elements in $B$ are $K$-linearly independent modulo $I^{2k}$.

We have $|\Gen(I^k)|={k+2\choose 2}$, since $\ell(I)=3$. Thus we conclude that
\[
\lambda(\widetilde{I^{2k}}/I^{2k})=\sum_{j=0}^{k-1} |\Gen(I^{2j})|=\sum_{j=0}^{k-1}{2j+2\choose 2}=\frac{2}{3}k^3+\frac{1}{2}k^2-\frac{1}{6}k.
\]
Since we know that $\lim_{k\to \infty} \length(\widetilde{I^{k}}/I^{k})/k^3$ exists, we see that
$$\lim_{k\to \infty} \lambda(\widetilde{I^{k}}/I^{k})/k^3=\lim_{k\to \infty} \lambda(\widetilde{I^{2k}}/I^{2k})/(2k)^3=1/12.$$
}
\end{Example}

\medskip
Theorem~\ref{small} has the following surprising consequences:

\begin{Corollary}
\label{surprising1}
Let $I\subset S=K[x_1,\ldots,x_n]$ be a graded ideal generated in degree $d$.
\begin{enumerate}
\item[{\em (a)}] If $\ell(I)<n$, then all generators of $\widetilde{I}$ are of degree $\geq d$.
\item[{\em (b)}] Suppose in addition that all powers of $I$ have a linear resolution, and that\\ $\depth S/I^r=0$ for some $r$. Then $\ell(I)=n$ and $\lim_{k\to\infty}\lambda(\widetilde{I^k}/I^k)/n^k\neq 0$.
\end{enumerate}
\end{Corollary}

\begin{proof}
(a) Suppose there exists  $g\in \widetilde{I}$ with $\deg g=c<d$. Then  $g^j(I^{k-j})_{(k-j)d}\subset \widetilde{I}^k$ for $j=1,\ldots,k$, and the elements in $g^j(I^{k-j})_{(k-j)d}$ are homogeneous of degree $jc+(k-j)d<kd$. It follows that $\lambda (\widetilde{I}^k/I^k)\geq \sum_{j=1}^k\dim_K (I^{k-j})_{(k-j)d}$. Thus we see that $\lambda (\widetilde{I}^k/I^k)$ grows like a polynomial of degree $\geq \ell(I)$, contradicting Theorem~\ref{small}(b).

(b) Assuming that $I^r$ has a linear resolution and that $\depth S/I^r=0$ implies that the $(n-1)$th syzygy module of $I^r$ has a generator of degree $rd+n-1$ which in turn implies that there is an element $f\in
\widetilde{I^r}$ of degree $rd-1$. It follows from part (a) that $\ell(I)=\ell(I^r)=n$. The proof of part (a) also shows that $\lambda (\widetilde{I}^k/I^k)\geq p(k)$, where $p$ is a polynomial of degree $\geq \ell(I)$. Since $\lim_{k\to\infty}\lambda(\widetilde{I^k}/I^k)/n^k$ exists it follows that $\deg p=\ell(I)$ and that $\lim_{k\to\infty}\lambda(\widetilde{I^k}/I^k)/n^k$ is greater than or equal to the  leading coefficient of $p$.
\end{proof}

 Theorem~\ref{small} can be used to derive the following  finiteness criterion. A  related result for ordinary symbolic powers was proven by Katz and Ratliff in Theorem A and Corollary~1 of \cite{KR}.

\begin{Theorem}
\label{strengthen}
Let $(R,\mm)$ be an excellent domain, and let $I$ and $J$ be proper ideals of $R$. Assume that $\depth R_P\geq 2$ for all $P\in V(J)$. Then the following conditions are equivalent:
\begin{enumerate}
\item[{\em (a)}]  $S_J(I)$ is finitely generated.
\item[{\em (b)}]  There exists an integer $r>0$ such that $\ell((I^r\: J^\infty)_P)<\dim R_P$ for all $P\in V(J)$.
\end{enumerate}
\end{Theorem}

\begin{proof} We use the criterion which says that $S_J(I)$ is finitely generated if and only if for some integer $d>0$ the  $d$th Veronese subalgebra $S_J(I)^{(d)}$ is standard graded, see for example \cite[Theorem 2.1]{HHT}

(a)\implies (b): We choose an integer $r$ such that $S_J(I)^{(r)}$ is standard graded. Then this implies that $(I^r\: J^\infty)^k=(I^{rk}\:J^\infty)$ for all $k$.  Hence if for a given   $P\in V(J)$ we set   $L=(I^r\: J^\infty)_P$, then it follows that all powers $L^k$ of $L$ are saturated in $R_P$, and assertion (b) is a consequence of the following claim: let $(R,\mm)$ be an excellent  local ring with $\depth R\geq 2$, and $I\subset R$ an ideal with the property that all powers of $I$ are saturated. Then $\ell(I)<\dim R$.

For the proof of the claim we view the Rees algebra  $R(I)=\Dirsum_kI^k t^k$ via the natural inclusion as a graded  subalgebra of the polynomial ring $R[t]$. We then get an short exact sequence of graded $R(I)$-modules
\[
0\To R(I)\To R[t]\To N\To 0 \quad \text{with} \quad N=R[t]/R(I),
\]
which induces the exact sequence
\[
H^0_\mm(N)\to H^1_\mm(R(I))\To H^1_\mm(R[t]).
\]
We notice that $H^0_\mm(N)=\Dirsum_k H^0_\mm(R/I^k)=0$ since all powers of $I$ are saturated, and that $H^1_\mm(R[t])=0$ since $\depth R\geq 2$. It follows that $H^1_{\mm R(I)}(R(I))=H^1_\mm(R(I))=0$. By \cite[Theorem 6.2.7]{BS}, this implies that  $\grade \mm R(I)\geq 2$. Therefore $\ell(I)=\dim R(I)/\mm R(I)\leq (\dim R+1)-2< \dim R$, as desired.

(b) \implies (a): Let $r>0$ be the  integer such that $\ell((I^r\: J^\infty)_P)<R_P$ for all $P\in V(J)$, and set $L=I^r\: J^\infty$. Then by Theorem~\ref{small}(a) we know that $S_J(L)$ is finitely generated. Thus there exists an integer $s>0$ such that $S_J(L)^{(s)}$ is standard graded. In other words, $(L^s\: J^\infty)^k=L^{ks}\: J^\infty$ for all $k$. Since $L=I^r\: J^\infty$ this is equivalent to saying that
\begin{eqnarray}
\label{any}
[(I^r\: J^\infty)^s\: J^\infty]^k=(I^r\: J^\infty)^{ks}\: J^\infty.
\end{eqnarray}
Now we  claim that for any two integers $i,j>0$ one has that $(I^i\: J^\infty)^j\:J^\infty=I^{ij}\: J^\infty$. The claim and (\ref{any}) then implies that $(I^{rs}\:J^\infty)^k=I^{rsk}\: J^\infty$ for all $k$. Hence $S_J(I)^{(d)}$ with $d=rs$ is standard graded, and so $S_J(I)$ is finitely generated.

In order to prove the claim first notice that   $I^i\subset I^i\: J^\infty$, so that $I^{ij}\subset (I^i\: J^\infty)^j$ and hence $I^{ij}\: J^\infty\subset (I^i\: J^\infty)^j\: J^\infty$. On the other hand, if $f\in (I^i\: J^\infty)^j\: J^\infty$, then $J^rf\in (I^i\: J^\infty)^j$ for  some $r>0$. Therefore there exist $g_1,\ldots, g_t\in I^i\: J^\infty$ and $c_{j_1,\ldots,j_t}\in R$ such that $J^rf= \sum c_{j_1,\ldots,j_t}g_1^{j_1}\cdots g_t^{j_t}$ where the sum is taken over all sequences  $(j_1,\ldots, j_t)$ of nonnegative integers with $j_1+j_2+\cdots +j_t=j$. For each $g_{j_k}$ there exists an integer $r_k>0$ such that $J^{r_k}g_{j_k}\in I^i$. Thus for a suitable big enough integer $\rho>0$ we get that $J^{\rho}f\in I^i$. In other words,  $f\in I^i\: J^\infty$.
\end{proof}

Theorem~\ref{strengthen} implies immediately the following result of Katz \cite{K} in case $R$ is excellent.

\begin{Corollary}
\label{particular}
The algebra  $\Dirsum_{k\geq 0}\widetilde{I^k}$ is finitely generated, if and only if  $\ell(\widetilde{I^r})<\dim R$ for some integer $r>0$.
\end{Corollary}

In general, $\Dirsum_{k\geq 0}\widetilde{I^k}$ is not finitely generated. Nevertheless we have

\begin{Theorem}
\label{nice}
Suppose that $(R,\mm)$ is a regular local ring of dimension $d$, which is essentially of finite type over a  field $K$ of characteristic zero. Suppose that $I\subset R$ is an ideal such that
the singular locus of $\mbox{Spec}(R/I)$ is $\mm$.  Then the limit
$$
\lim_{k\rightarrow\infty} \frac{\lambda(\widetilde{I^k}/I^k)}{k^d}\in\RR
$$
exists.
\end{Theorem}

\begin{proof} If $I$ is $\mm$-primary, then $\widetilde{I^k}=R$ for all $k$,
and thus $\lambda(\widetilde{I^k}/I^k)=\lambda(R/I^k)$ is a polynomial in $k$ of
degree $d$ for $k\gg0$. Thus the limit exists.

Now assume that $I$ is not $\mm$-primary. Since $R$ is regular and the singular locus of $\mbox{Spec}(R/I)$ is $\mm$, we have that $I_p^k$ is a complete ideal in $R_p$ for all $p\in\mbox{Spec}(R)-\{\mm\}$ and $k\ge 0$. Thus
$$
\widetilde{I^k}=I^k:\mm^{\infty}=\overline{I^k}:\mm^{\infty}\supset \overline{I^k}
$$
for all $k>0$.

Consider the exact sequence of finite length $R$-modules
$$
0\rightarrow \overline{I^k}/I^k\rightarrow \widetilde{I^k}/I^k\rightarrow \widetilde{I^k}/\overline{I^k}\rightarrow 0.
$$
Since $\Dirsum_{k\ge 0}\overline{I^k}$ is a finitely generated $\Dirsum_{k\ge 0}I^k$-module, the quotient
$\Dirsum_{k\ge 0}\overline {I^k}/I^k$ is a finitely generated $\Dirsum_{k\ge 0}I^k$ module, which is annihilated by $\mm^r$ for some $r$. Thus
$\lambda(\overline{I^k}/I^k)$ is a polynomial of degree $\le d-1$ for $k\gg 0$,
and we have reduced to showing that
$$
\lim_{k\rightarrow \infty}\frac{\lambda(\widetilde{I^k}/\overline{I^k})}{k^d}
$$
exists.

The blow up
$\mbox{Proj}(\Dirsum_{k\ge 0}I^k)$ of $I$ is nonsingular away from the fiber over the
maximal ideal $\mm$ of $R$. Let $Y\rightarrow \mbox{Proj}(\Dirsum_{k\ge 0}I^k)$
be a resolution of singularities which is an isomorphism away from the
fiber over the maximal ideal $\mm$ of $R$. Let $X=\mbox{Spec}(R)$, and $f:Y\rightarrow X$
be the natural map. Let ${\mathcal L}=I{\mathcal O}_Y$.
${\mathcal L}={\mathcal O}_Y(-F-E)$ where $E$ is an effective divisor such that $f(E)=\mm$, and $F$ is a
reduced effective divisor whose components are the prime divisors corresponding
to the $P_P$-adic valuations in $R_P$, where $P$ ranges over the minimal primes
of $I$.

For $k\in \NN$, we have exact sequences
$$
0\rightarrow {\mathcal O}_Y(-kE)\rightarrow{\mathcal O}_Y\rightarrow {\mathcal O}_{kE}\rightarrow 0.
$$
Tensoring with ${\mathcal O}_Y(-kF)$, we have exact sequences
$$
0\rightarrow {\mathcal L}^k\rightarrow {\mathcal O}_Y(-kF)\rightarrow {\mathcal O}_{kE}(-kF)\rightarrow 0,
$$
where ${\mathcal O}_{kE}(-kF)$ denotes the invertible sheaf ${\mathcal O}_{kE}\otimes{\mathcal O}_Y(-kF)$ on the scheme $kE$.
Taking global sections, we have  exact sequences
\begin{equation}\label{eq2}
0\rightarrow \overline{I^k}\rightarrow \widetilde{I^k}\rightarrow H^0(kE,{\mathcal O}_{kE}(-kF))
\rightarrow H^1(Y,{\mathcal L}^k).
\end{equation}

Let $N=\Dirsum_{k\ge 0}H^1(Y,{\mathcal L}^k)$. Then $N$ is naturally a $\Dirsum_{k\ge0}\overline{I^k}\cong \Dirsum_{k\ge 0}H^0(Y,{\mathcal L}^k)$ module. We will show that $N$ is a
finitely generated $\Dirsum_{k\ge 0}\overline{I^k}$-module.

Let $Z=\mbox{Proj}(\Dirsum_{\ge 0}\overline{I^k})$, and ${\mathcal N}=I{\mathcal O}_Z$. Since $Y$ is normal and dominates the blowup of $I$, the map $f:Y\rightarrow X$
 factors as
$$
Y\stackrel{g}{\rightarrow} Z\stackrel{h}{\rightarrow} X.
$$
From the first terms of the Leray spectral sequence, we have an exact sequence
\begin{equation}\label{eq4}
0\rightarrow H^1(Z,g_*({\mathcal L}^k))\rightarrow H^1(Y,{\mathcal L}^k)
\rightarrow H^0(Z,R^1g_*({\mathcal L}^k)).
\end{equation}
We have $g_*({\mathcal L}^k)\cong {\mathcal N}^k$  and
$R^1g_*({\mathcal L}^k)\cong {\mathcal N}^k\otimes R^1g_*{\mathcal O}_Y$
since $Z$ is normal, and by the projection formula.
Since ${\mathcal N}$ is ample on $Z$, there exists $k_0$ such that
$H^1(Z,g_*({\mathcal L}^k))=0$ for $k\ge k_0$. Since $h$ is proper,
$H^1(Z,{\mathcal N}^k)$ is a finitely generated $R$-module for all $k$. Thus
$\Dirsum_{k\ge 0}H^1(Z,g_*({\mathcal L}^k))$ is a finitely generated
$\Dirsum_{k\ge 0}\overline{I^k}=\Dirsum_{k\ge 0}H^0(Z,{\mathcal N}^k)$-module.
Since $g$ is proper, $R^1g_*{\mathcal O}_Y$ is a coherent ${\mathcal O}_Z$
module. Since $\mathcal N$ is ample, there exists $s\in\NN$ and $a_i\in\ZZ$ such that there is a surjection
$$
\Dirsum_{i=1}^s{\mathcal N}^{a_i}\rightarrow R^1g_*{\mathcal O}_Y
$$
of ${\mathcal O}_Z$ modules. Let ${\mathcal K}$ be the kernel of this map. We have an exact sequence
$$
\Dirsum_{n\ge 0}(\Dirsum_{i=1}^sH^0(Z,{\mathcal N}^{n+a_i}))\rightarrow \Dirsum_{n\ge 0}H^0(Z,{\mathcal N}^n\otimes R^1g_*{\mathcal O}_Y)\rightarrow \Dirsum_{n\ge 0}H^1(Z,{\mathcal K}\otimes{\mathcal N}^n).
$$
Since ${\mathcal N}$ is ample, there exists an $n_0$ such that
$H^1(Z,{\mathcal K}\otimes{\mathcal N}^n)=0$ for $n\ge n_0$. Since $R$ is normal,
we have
$$
H^0(Z,{\mathcal N}^i)=\left\{\begin{array}{ll}
R&\mbox{ if }i\le 0\\
\overline{I^i}&\mbox{ if }i>0.
\end{array}\right.
$$
Thus $\Dirsum_{n\ge 0}(\Dirsum_{i=1}^s H^0(Z,{\mathcal N}^{n+a_i})$
is a finitely generated $\Dirsum_{n\ge 0}H^0(Z,{\mathcal N}^n)$ module.
Since $H^1(Z,{\mathcal K}\otimes{\mathcal N}^n)$ are finitely generated $R$-modules for all $n$, which are zero for $n\ge n_0$, it follows that
$\Dirsum_{n\ge 0}H^0(Z,{\mathcal N}^n\otimes R^1g_*{\mathcal O}_Y)$ is a finitely generated $\Dirsum_{n\ge0}H^0(Z,{\mathcal N}^n)$ module.
From (\ref{eq4}), we see that $N$ is a finitely generated $\Dirsum_{k\ge 0}\overline{I^k}$-module.

Since $N$ is a finitely generated $\Dirsum_{k\ge 0}\overline{I^k}$-module,
and the support of $H^1(Y,{\mathcal L}^k)$ is contained in $\{\mm\}$ for all $k$,
there exists a positive integer $r$ such that $\mm^rN=0$. Since
$$
\dim(\Dirsum_{k\ge 0} \overline{I^k})/\mm(\Dirsum_{k\ge 0}\overline I^k)
\le \mbox{dim }R=d,
$$
there exists a constant $c$ such that
$\lambda(H^1(Y,{\mathcal L}^k))\le ck^{d-1}$ for all $k$. From comparison with (\ref{eq2}),
we have reduced to showing that
\begin{equation}\label{eq5}
\lim_{k\rightarrow \infty}\frac{\lambda(H^0(kE,{\mathcal O}_{kE}(-kF)))}{k^d}
\end{equation}
 exists.

If $R/\mm$ is algebraic over $K$, let $K'=K$. If $R/\mm$ is transcendental over $K$, let $t_1,\ldots,t_r$ be a lift of a transcendence basis of $R/\mm$ over $K$ to $R$. The rational function field $K(t_1,\ldots,t_r)$ is contained in $R$.
Let $K'=K(t_1,\ldots,t_r)$. We have that
$R/\mm$ is finite algebraic over $K'$. There exists a nonsingular affine $K'$-variety $U$ such that $R$ is the local ring of a closed point $\alpha$ of $U$.
Let $\overline X$ be a nonsingular projective closure of $U$, and let
${\mathcal I}$ be an extension of $I$ to an ideal  sheaf on $\overline X$.
Let $\overline f:\overline Y\rightarrow \overline X$ be a resolution of singularities such that $\overline Y\rightarrow \overline X$ factors through the blow up of ${\mathcal I}$, and $\overline f^{-1}(X)\cong Y$. We may identify $E$
(and $kE$ for all positive integers $k$) with a closed subscheme of $\overline Y$. Let $\overline F$ be the Zariski closure of
$F$ in $\overline Y$. Since the singular locus of $\mbox{Spec}(R/I)$ is $\mm$,
we may choose $U$, $\overline X$ and $\overline Y$ so that the singular locus of the scheme
$\mbox{Spec}({\mathcal O}_{\overline X}/{\mathcal I})$  is the isolated point $\alpha$, and $\overline f_*{\mathcal O}_Y(-\overline F)\cong {\mathcal I}_{\beta}$ for $\beta\in\overline X-\{\alpha\}$.

There exists a line bundle $\mathcal M$ on $\overline Y$ such that
$\mathcal M\otimes{\mathcal O}_Y\cong{\mathcal L}$ is generated by global sections and is big. We can construct $\mathcal M$ by taking any ample line
bundle $\mathcal A$ on $\overline X$, and taking
$$
{\mathcal M}=\overline f^*({\mathcal A}^t)\otimes {\mathcal I}{\mathcal O}_{\overline Y}
$$
for
$t$ sufficiently large. Let ${\mathcal B}={\mathcal M}\otimes{\mathcal O}_Y(E)$. We have an exact sequence
$$
0\rightarrow {\mathcal O}_{\overline Y}(-kE)\rightarrow {\mathcal O}_{\overline Y}\rightarrow {\mathcal O}_{kE}\rightarrow 0.
$$
Tensoring with ${\mathcal B}^k$, we have  exact sequences
$$
0\rightarrow {\mathcal M}^k\rightarrow {\mathcal B}^k\rightarrow {\mathcal O}_{kE}(-kF)\rightarrow 0.
$$
Taking global sections, we have exact sequences
\begin{equation}\label{eq3}
0\rightarrow H^0(\overline Y,{\mathcal M}^k)\rightarrow H^0(\overline Y,{\mathcal B}^k)\rightarrow H^0(kE,{\mathcal O}_{kE}(-kF))
\rightarrow H^1(\overline Y,{\mathcal M}^k).
\end{equation}

Since ${\mathcal M}$ is semiample (generated by global sections and big), we have that
$$
\lim_{k\rightarrow\infty}\frac{h^1(\overline Y,{\mathcal M}^k)}{k^d}=0
$$
(for instance as a special case of [F1] or by consideration of the Leray spectral sequence of the mapping from $\overline Y$ given by  the global sections of a high power of ${\mathcal M}$). Further,
$\Dirsum_{n\ge 0}H^0(\overline Y,{\mathcal M}^k)$ is a finitely generated $K'$-algebra of dimension $d+1$. Thus

$$
\lim_{k\rightarrow\infty}\frac{h^0(\overline Y,{\mathcal M}^k)}{k^d}\in \QQ
$$
exists.
Since ${\mathcal B}$ is big,
by the  corollary given in [L] or [CHST] to [F2], we have that
$$
\lim_{k\rightarrow\infty}\frac{h^0(\overline Y,{\mathcal B}^k)}{k^d}\in\RR
$$
exists. From the sequence (\ref{eq3}), we see that
$$
\lim_{k\rightarrow\infty}\frac{h^0(kE,{\mathcal O}_{kE}(-kF))}{k^d}\in\RR,
$$
and the conclusions of the theorem now follow, by applying the formula
$$
h^0(kE,{\mathcal O}_{kE}(-kF))=\mbox{dim}_{K'}H^0(kE,{\mathcal O}_{kE}(-kF))
=[R/\mm:K']\lambda(H^0(kE,{\mathcal O}_{kE}(-kF)))
$$
to equation (\ref{eq5}).
\end{proof}

\begin{Corollary} Suppose that $R=K[[x_1,\ldots,x_n]]$ is a power series ring
over a field $K$ of characteristic zero, and $I\subset R$ is an equidimensional ideal such that the singular locus of $\Spec(R/I)$ is $\mm=(x_1,\ldots,x_n)$.  Then the limit
$$
\lim_{k\rightarrow\infty} \frac{\lambda(\widetilde{I^k}/I^k)}{k^d}\in\RR
$$
exists.
\end{Corollary}
\begin{proof}
By \cite[Theorem 1]{H} or \cite[Theorem A]{CS}, there exists an ideal $J\subset K[x_1,\ldots,x_n]$ and a $K$-algebra isomorphism $\phi:R\rightarrow R$ such
that $\phi(I)=JR$. Thus $\phi(I^k)=J^kR$ and $\phi(\widetilde{I^k})=\widetilde{J^k}R=\widetilde{J^kR}$ for all $k$. We have that
$$
\lambda(\widetilde{I^k} /I^k)=\lambda(\widetilde{J^k}R/J^kR)=\lambda(\widetilde{J^k}/J^k)
$$
for all $k\in\NN$. Now by Theorem \ref{nice},
$$
\lim_{k\rightarrow\infty} \frac{\lambda(\widetilde{I^k}/I^k)}{k^d}
=\lim_{k\rightarrow\infty} \frac{\lambda(\widetilde{J^k}/J^k)}{k^d}
\in\RR.
$$
\end{proof}

\section{Form ideals of powers of complete intersections}

Let $(R,\mm)$ be a Noetherian local ring, $I\subset \mm$ an ideal. For any local ring $(S,\nn)$ we denote by $G(S)=\Dirsum_{k\geq 0}\nn^k/\nn^{k+1}$ the associated graded ring  of $S$. The canonical epimorphism $R\to R/I$ induces an epimorphism $G(R)\to G(R/I)$ whose kernel we denote by $I^*$. The graded ideal $I^*$ is called the form ideal of $I$.  If $f\in R$ and $d$ is the maximal number such that $f\in\mm^d$, then we set $f^*=f+\mm^{d+1}$ and call it the leading form of $f$. The leading forms $f^*$ with $f\in I$ generate $I^*$. Any system of generators $f_1,\ldots,f_m$ of $I$ such that $f_1^*,\ldots,f_m^*$ generates $I^*$ is called a standard basis of $I$. A standard basis of $I$ is a system of generators of $I$, but  is usually not a minimal system of generators.

\medskip
The following lemma is well-known. For the convenience of the reader we give a sketch of its proof.

\begin{Lemma}
\label{lifting}
Let $(R,\mm)$ be a local ring such that $G(R)$ is domain, and  $I\subset R$ an ideal. Let $f_1,\ldots, f_m$ be a system of generators of $I$. Then $f_1,\ldots, f_m$ is a standard basis of $I$ if  all relations of $f_1^*,\ldots, f_m^*$ can be lifted. In other words, whenever there is  a homogeneous relation
\[
g_1f_1^*+g_2f_2^*+\cdots +g_mf_m^*=0,
\]
with $g_i\in G(R)$,
then there exist $h_i\in R$ with $g_i=h_i^*$ for $i=1,\ldots,m$ such that
\[
h_1f_1+h_2f_2+\cdots +h_mf_m=0.
\]
Moreover it is sufficient to test the lifting property for a system of homogeneous generators of the  relation module of $f_1^*,\ldots, f_m^*$.
\end{Lemma}

\begin{proof}
Let $f\in I$. Then $f=\sum_{i=1}^nc_if_i$ with $c_i\in R$. Let $d_0=\min\{\deg(c_if_i)^* \: i=1,\ldots,n\}$. It is clear that $d_0\leq d=\deg f^*$. Assume that $d_0<d$, and let ${\mathcal I}$ be the set of integers $i$ with $\deg (c_if_i)^*=d_0$. Since $G(R)$ is a domain, it follows that $(c_if_i)^*=c_i^*f_i^*$ for all $i$, and since $d_0<d$, we see that $\sum_{i\in\mathcal I}c_i^*f_i^*=0$. By the lifting property there exist $h_i\in R$ with $h_i^*=c_i^*$ for all $i\in \mathcal I$ and such that $\sum_{i\in \mathcal I}h_if_i=0$. Thus we get a new presentation $f=\sum_{i=1}^nc_i'f_i$,  where $c_i'=c_i-h_i$ for $i\in{\mathcal I}$, and $c_i'=c_i$ for $i\not\in {\mathcal I}$. Since $\deg (c_i')^*>\deg c_i^*$ for $i\in {\mathcal I}$, we conclude that $\min\{\deg(c_i'f_i)^* \: i=1,\ldots,n\}>d_0$. Thus in a finite number of steps we arrive at a presentation $f=\sum_{i=1}^nb_if_i$ with $\deg (b_if_i)^*\geq d$ for all $i$. Then $f^*=\sum_{i\in {\mathcal J}} b_i^*f_i^*$ where $\mathcal J=\{i\: \deg (b_if_i)^*= d\}$

Let $g_i=(g_{i1},\ldots,g_{in})$, $i=1,\ldots,m$ be a system of homogeneous generators of  the relation module of $f_1^*,\ldots, f_m^*$ with the property that each $g_i$ can be lifted to a relation $h_i=(h_{i1},\ldots,h_{in})$ of $f_1,\ldots, f_m$, and let $g$ be an arbitrary homogeneous relation of $f_1^*,\ldots, f_m^*$. Then there exist elements $c_i\in R$ such that $g=\sum_{i=1}^mc_i^*g_i$ and with the property that $\deg c_i^*+\deg g_i=\deg g$. Recall that $\deg g_i=\deg g_{ij}+\deg f_j^*$ for all $j$. Keeping this in mind one sees that $h=\sum_{i=1}^mc_ih_i$ is a lifting of $g$.
\end{proof}

We are interested in the algebra $A(I)=\Dirsum_{k\geq 0}(I^k)^*$. If this  algebra happens to be finitely generated, then for $k\gg 0$ the  coefficients of the Hilbert polynomials  $P_{R/I^k}(t)$ are quasi-polynomials as functions of $k$.

For the formulation of the next result we need the following definition: let $A$ be graded $K$-algebra and $J\subset A$ a graded ideal. We say that $J$ is liftable, if there exists a graded ideal $\tilde{J}\subset A[t_1,\ldots,t_r]$ in a polynomial ring  extension $A[t_1,\ldots,t_r]$ of $A$ with $\deg t_i>0$ satisfying the following properties: (i) $\tilde{J}$  is generically a complete intersection, (ii)   $t_1,\ldots,t_r$ is a regular sequence on $A[t_1,\ldots,t_r]/\tilde{J}$,  and  (iii) $A[t_1,\ldots,t_r]/\tilde{J}$ modulo $(t_1,\ldots,t_r)$ is isomorphic to $A/J$. If $J$ is liftable to $\tilde{J}$, then $J$ is called a specialization of $\tilde{J}$.

\begin{Theorem}
\label{complete}
Let $(R,\mm)$ be a local ring such that $G(R)$ is a domain, and $I\subset R$ an ideal. Then
$A(I)=\Dirsum_{k\geq 0}(I^k)^*$ is standard graded, if $I^*$ is a complete intersection.

Conversely, suppose that $G(R)$ is Cohen-Macaulay and  $I\subset R$ is a complete intersection ideal, satisfying:
\begin{enumerate}
\item[{\em (i)}] $\dim R/I=0$,
\item[{\em (ii)}] $(I^*)^k=(I^k)^*$ for infinitely many  $k>1$ (for example, if $A(I)$ is standard graded),
\item[{\em (iii)}]$I^*$ is liftable.
\end{enumerate}
Then $I^*$ is a complete intersection.
\end{Theorem}

\begin{proof}
Let $g_1,\ldots,g_m$ be the regular sequence generating $I^*$, and let $f_1,\ldots, f_m\in I$ with $f_i^*=g_i$ for $i=1,\ldots,m$. Then $f_1,\ldots,f_m$ is a regular sequence generating $I$, and in particular it is a standard basis of $I$. Now fix an integer $k>1$. We claim that the monomials $f^a=f_1^{a_1}f_2^{a_2}\cdots f_m^{a_m}$ in $f_1,\ldots,f_m$ of degree $k$ form a standard basis of $I^k$. This will then imply that $(I^k)^*=(I^*)^k$ for all $k$, so that $A(I)$ is standard graded.

In order to see that the monomials $f^a$ of degree $k$ form indeed a standard basis of $I^k$ we just need to show that all generating relations of the ideal generated  by the  leading forms of the elements $(f^a)^*$ with $a\in \NN^m$ and   $|a|=k$ can be lifted. Observe that $(f^a)^*=g_1^{a_1}g_2^{a_2}\cdots g_m^{a_m}$. Since $g_1,\ldots,g_m$ is a regular sequence, the relation module of $(g_1,\ldots,g_m)^k$ is generated by  relations of the form
\[
g_j(g_1^{a_1}g_2^{a_2}\cdots g_i^{a_i+1}\cdots g_m^{a_m})-g_i(g_1^{a_1}g_2^{a_2}\cdots g_j^{a_j+1}\cdots g_m^{a_m}).
\]
These relations can obviously all be lifted.

For the second part of the theorem let $J=I^*$. Since $I$ is a complete intersection ideal with $\dim R/I=0$, (ii) implies that for infinitely many  integers $k>1$ we have
\begin{eqnarray}
\label{kpower}
\lambda(G(R)/J^k)&=&\lambda(G(R)/(I^k)^*)=\lambda(R/I^k)\\
&=&{d+k-1 \choose k-1}\lambda(R/I)={d+k-1 \choose k-1}\lambda(G(R)/J),\nonumber
\end{eqnarray}
where $d=\dim G(R)$.

Let $\tilde{J}\subset G(R)[t_1,\ldots,t_r]$ be a lifting of $J$.  Since $\tilde{J}$ is generically a complete intersection, the associativity formula for multiplicities (\cite[Corollary 4.7.8]{BH})  implies that
\begin{eqnarray}
\label{lifted}
e(G(R)[t_1,\ldots,t_r]/\tilde{J}^k)={d+k-1 \choose k-1}e(G(R)[t_1,\ldots,t_r]/\tilde{J}).
\end{eqnarray}
Since  $e(G(R)[t_1,\ldots,t_r]/\tilde{J})=\lambda(G(R)/J)$, the equations (\ref{kpower}) and (\ref{lifted}) imply that $$e(G(R)[t_1,\ldots,t_r]/\tilde{J}^k)=\lambda(G(R)/J^k)$$ for infinitely many $k>1$ which in turn implies that $G(R)[t_1,\ldots,t_r]/\tilde{J}^k$ is Cohen--Macaulay, see
\cite[Corollary 4.7.11]{BH}. Thus by a result of Cowsik and Nori \cite{CN} and its  generalization by Waldi \cite[Korollar 1]{W} it follows that $\tilde{J}$ is a complete intersection.  Hence $I^*=J$ is a complete intersection as well.
\end{proof}

\begin{Corollary}
\label{application}
Let $(R,\mm)$ be a regular local ring and  $I\subset R$ be a complete intersection ideal with $\dim R/I=0$. Assume further that either  $(R,\mm)$ is 2-dimensional or $I^*$ is a monomial ideal. Then the following conditions are equivalent:
\begin{enumerate}
\item[{\em (a)}] $A(I)$ is standard graded;
\item[{\em (b)}] $I^*$ is a complete intersection;
\item[{\em (c)}] for infinitely many integers  $k$ we have  $(I^k)^*=(I^*)^k$.
\end{enumerate}
\end{Corollary}

\begin{proof} Under the given assumptions the ideal $I^*$  is liftable. Indeed, if $I^*$ is a monomial ideal, then one applies polarization, see \cite[Lemma 4.2.16]{BH}, and if $\dim R=2$, then $I^*$   is perfect of codimension 2. Hence if $I^*$ is generated by $m$ elements, the Hilbert--Burch theorem \cite[Theorem 1.4.17]{BH} implies that $I^*$ is the specialization  of the ideal of maximal minors of an $m\times(m+1)$-matrix of indeterminates. This ideal is generically a complete intersection.

Now we see that Theorem~\ref{complete} yields the implications (c)\implies (b) and (b)\implies (a). The implication (a)\implies (c) is trivial.
\end{proof}

It would be interesting to know whether for a complete intersection  the  conditions (a), (b) and (c) in Corollary~\ref{application} are equivalent without the assumption that $I^*$ is liftable.

\medskip
The following simple example shows that even for a complete intersection the algebra $A(I)$  need not to be finitely generated.

\begin{Example}
\label{infinite}
{\em Let $K$ be a field and consider the ideals $I=(x^2,y^3-xy)\subset S=K[[x,y]]$. We claim that
\begin{eqnarray}
\label{power}
(I^k)^*=((xy,x^2)^k, \{x^iy^{4k-3i+1}\}_{i=0,\ldots,k-1}).
\end{eqnarray}
The claim implies that $y^{4k+1}$ is a minimal generator of $(I^k)^*$. It follows that for each $k$, the element $y^{4k+1}\in (I^k)^*$ is a minimal generator of degree $k$ of the form algebra $A=\Dirsum_{k\geq 0}(I^k)^*$ of $I$. It particular, we see that $A$ is not finitely generated.

We prove (\ref{power}) by induction of $k$, and  set $f=x^2$ and $g=y^3-xy$. In order to prove  (\ref{power})for $k=1$  we first notice that $y^5\in I$. Indeed, we have
$y^5=(y^3+x)g+yf$.
It follows that $(x^2,xy,y^5)\subset I^*$. Applying the Buchberger criterion we see immediately that $x^2,xy-y^3,y^5$ is a Groebner basis of $I$ with respect to the lexicographical order. Hence  $\ini(I)=(x^2,xy,y^5)$ is the initial ideal of $I$ with respect to this monomial order. Therefore $\lambda(S/(x^2,xy,y^5))=\lambda(S/I)$. On the other hand we have $\lambda(S/I^*)=\lambda(S/I)$. This implies that $(x^2,xy,y^5)= I^*$ and proves the claim for $k=1$. It also shows that $\lambda(S/I)=6$.

Now let $k>1$ and assume that (\ref{power}) holds for all $j<k$. Then
\begin{eqnarray}
\label{product}
I^*(I^{k-1})^*+(y^{4k+1})=((xy,x^2)^k, \{x^iy^{4k-3i+1}\}_{i=0,\ldots,k-1})\subset (I^k)^*.
\end{eqnarray}
Thus it remains to be shown that
\begin{eqnarray}
\label{res}
(I^k)^*=I^*(I^{k-1})^*+(y^{4k+1}).
\end{eqnarray}
Since $I$ is generated by a regular sequence it follows that all the modules $I^j/I^{j+1}$ are free $S/I$-modules of rank $j+1$. From this we deduce that
\[
\lambda(S/(I^k)^*)=\lambda(S/(I^k))=\lambda(S/I){k+1\choose 2}=6{k+1\choose 2}.
\]
Now we compute the length of $S/I^*(I^{k-1})^*+(y^{4k+1})$. In view of formula (\ref{product}) we see that $S/I^*(I^{k-1})^*$ has the following monomial $K$-basis: $C\union \Union_{i=0}^{k-1}B_i$ where
\[
C=\{x^k(x^{i}y^j)\}_{i+j\leq k-1}, \quad  \text{and}\quad
B_i=\{x^iy^j\}_{j\leq 4k-3i}.
\]
Counting the number of elements of this basis we see that
\begin{eqnarray*}
\lambda(S/I^*(I^{k-1})^*)+(y^{4k+1}))&=&|C|+\sum_{i=0}^{k-1}|B_i|\\
&=&{k+1\choose 2}+\sum_{i=0}^{k-1}(4k-3i+1)=6{k+1\choose 2}.
\end{eqnarray*}
Thus
\begin{eqnarray}
\label{equal}
\lambda(S/I^*(I^{k-1})^*+(y^{4k+1}))=\lambda(S/(I^k)^*),
\end{eqnarray}
and hence it suffices to show that $y^{4k+1}\in I^k$. Indeed, we will show that for $i=0,1,\ldots, 2k$  the monomials $x^{2k-i}y^{2i+1}$ belong to $I^k$. We proceed by induction on $i$. For $i=0$ we have $x^{2k}y= f^ky\in I^k$. Now let $i>0$ and suppose that $x^{2k-j}y^{2j+1}\in I^k$ for $j<i$. Let the integers $a$ and $b$ be defined by the equations
\[
2k-i=2a+r_1,\quad 0\leq r_1\leq 1,\quad \text{and}\quad 2i+1=3b+r_2,\quad 0\leq r_2\leq 2.
\]
Then $4k+1=4a+3b+2r_1+r_2$ which implies that $4a+3b\geq 4k-3$. From this we deduce that $a+b\geq k$.  Therefore $x^{r_1}y^{r_2}(x^2)^a(y^3-xy)^b\in I^k$ and $x^{2k-i}y^{2i+1}-x^{r_1}y^{r_2}(x^2)^a(y^3-xy)^b$ is a linear combination of monomials of the form $x^{2k-j}y^{2j+1}$ with $j<i$. Since by induction hypothesis these monomials belong to $I^k$, we conclude that $x^{2k-i}y^{2i+1}\in I^k$.
}
\end{Example}

We can slightly modify Example~\ref{infinite} to get finite generation of the algebra $A(I)$ depending on the characteristic of the base field.

\begin{Example}
\label{depending}
{\em Let $I=(x^2+y^2, (x+y)y+y^3)\subset K[[x,y]]$. Then the algebra $A(I)= \Dirsum_{k\geq 0}(I^k)^*$ is standard graded if $\chara K\neq 2$, and it is not finitely generated if $\chara K=2$. Indeed, if $\chara K\neq 2$, then the leading forms $f^*=x^2+y^2$ and $g^*=(x+y)y$ of $f=x^2+y^2$ and $g=(x+y)y+y^3$ are prime to each other. Hence Theorem~\ref{complete}
implies that $A(I)$ is standard graded. On the other hand, if $\chara K=2$, then $f^*=(x+y)^2$. Applying the linear automorphism $\varphi \: K[[x,y]]\to K[[x,y]]$ with $\varphi(x)=x+y$ and $\varphi(y)=y$ we see that $\varphi(I)=(x^2,xy+y^3)$. It follows that $A(I)$ is not finitely generated since $A(\varphi(I))$ is not finitely generated.
}
\end{Example}

At present we do not know of a complete intersection ideal $I$ for which $A(I)$ is finitely generated but not standard graded.

\medskip
We now describe the relationship between form ideals of powers and symbolic powers. We fix a field $K$ and consider an ideal $I \subset R=K[[x_1,\ldots,x_n]]$. The following lemma establishes the link between the two concepts.

\begin{Lemma}
\label{phi}
Let $\alpha\: R\to R[[s]]$ be the $K$-algebra homomorphism with $\alpha(x_i)=x_is$ for $i=1,\ldots, n$. We denote by $\alpha(I)$ the ideal in $R[[s]]$ generated by the elements $\alpha(f)$ with $f\in I$, and set  $I^\sharp=\alpha(I)\: s^\infty$.  Then $s$  is a regular element on $A=R[[s]]/I^\sharp$ and $A/(s)\iso R/I^*R$.
\end{Lemma}

\begin{proof}
We first observe that $s$ is a regular element on $A=R[[s]]/I^\sharp$. Indeed, if $sf\in I^\sharp$ for some $f\in  R[[s]]$, then there exists exists an integer $k$  such that $s^{k+1}f=s^k(sf)\in \alpha(I)$. Then $f\in I^\sharp$.

In order to prove the isomorphism  $A/(s)\iso R/I^*R$, we show that $I^\sharp$ is generated by the elements $f^\sharp$ with $f\in I$, where for $f=f_d+f_{d+1}+\cdots  \in R$ with each $f_i$ homogeneous of degree $i$ and $f_d\neq 0$, we set $f^\sharp=f_d+sf_{d+1}+\cdots +s^{i-d}f_i+\cdots$. Observe that $f^\sharp=s^{-d}\alpha(f)$,  where  $d$ is the initial  degree of $f$. This shows that $f^\sharp\in I^\sharp$ for all $f\in I$.

Conversely, let $f\in I^\sharp$. Then there exists an integer $k$ such that $s^kf\in \alpha(I)$. Assigning to $s$ the degree $-1$ and to each $x_i$ the degree $1$, we see that the generators of $\alpha(I)$ are homogeneous of degree $0$. Here we call a power series $h\in R[[s]]$ homogeneous, if all monomials in the support of $h$ are of same degree. Let $h\in R[[s]]$ be a power series, and $i\in \ZZ$. We let $h_i$ be the sum of those  terms in $h$ whose degree is $i$. Then $h$ is the formal sum of the $h_i$  and each $h_i$ is homogeneous  of degree $i$. We call $h_i$ the $i$th homogeneous component of $h$. The expression $h=\sum_ih_i$ makes sense, because the monomial support of $h_i$ and $h_j$ is disjoint for $i\neq j$. Suppose now that $g$ is homogeneous of degree $j$. Then $(hg)_i =h_{i-j}g$ for all $j$. Thus if $J\subset R[[s]]$ is an ideal generated by homogeneous elements $g_1,\ldots, g_r$ of degree $j_1,\ldots, j_r$, respectively, then $h$ belongs to $J$ if
and only if all its homogeneous components belong to $J$, as is the case for a positively graded algebra. Indeed, if $h=\sum_{i=1}^ra_ig_i$, then the $i$th homogeneous component of $h$ is $\sum_{i=1}^r(a_i)_{i-j_i}g_i$, and thus it belongs to $J$.

Hence, since $\alpha(I)$ is generated by homogeneous elements,  we may assume $s^kf$ is homogeneous. In particular  $f$ is homogeneous,  say $\deg f=d$. Hence  there exist $l_1,\ldots, l_r\in I$ such that  $s^kf=\sum_{i=1}^rg_i\alpha(l_i)$ with each $g_i\in R[[s]]$  homogeneous of degree $d_i$ and such that $d=d_i+k$  for $i=1,\ldots,r$.

Since $g_i$ is homogeneous, it is of the form $\sum_{j\geq 0}g_{ij}s^j$ where for all $j$, $g_{ij}$ is a homogeneous polynomial in the variables $x_1,\ldots, x_n$ of degree $d_i+j$. Let $h_i=\sum_{j\geq 0}g_{ij}$. Then there exist integers $k_i$ such that $g_i=s^{k_i}\alpha(h_i)$ for $i=1,\ldots,r$. We have $k_i=-d_i$ for all $i$, since $\deg \alpha(h_i)=0$. It follows that
\[
s^kf=\sum_{i=1}^rs^{-d_i}\alpha(b_i) \quad \text{with}\quad b_i=h_il_i\in I.
\]
Write $f=\sum_lf_ls^l$ where each $f_l$ is a homogeneous polynomial of degree $d-l$  in the variables $x_1,\ldots, x_n$, and  write each $b_i=\sum_jb_{ij}$, where $b_{ij}$ is a homogenous polynomial of degree $j$  in the variables $x_1,\ldots, x_n$. Then we get
\[
s^kf=\sum_lf_ls ^{l+k}=\sum_{i=1}^r(\sum_jb_{ij}s^{j-d_i})=\sum_{i=1}^r(\sum_jb_{ij}s^{j+k-d})=
\sum_j(\sum_{i=1}^rb_{ij})s^{j+k-d}.
\]
Comparing coefficients we see that $f_l=\sum_{i=1}^rb_{i,l+d}$ for all $l$. This shows that $\sum_lf_l=\sum_{i=1}^rb_i\in I$. Thus if we set $g=\sum_{i=1}^rb_i$, then $g\in I$ and   $f=s^mg^\sharp$ for some nonnegative  integer $m$, as desired.
\end{proof}

We set  $J=(s)$. Then  we get

\begin{Proposition}
\label{comparison}
The following conditions are equivalent:
\begin{enumerate}
\item[{\em (a)}] $S_J(I^\sharp)=\Dirsum_{k\geq 0}(I^\sharp)^k:s^\infty$ is finitely generated (resp.\ standard graded).
\item[{\em (b)}] $A(I)=\Dirsum_{k\geq 0}(I^k)^*$ is finitely generated (resp.\ standard graded).
\end{enumerate}
\end{Proposition}

\begin{proof}    We first notice that
\begin{eqnarray}
\label{equal1}
(I^\sharp)^k\: s^\infty=(\alpha(I)\: s^\infty)^k\: s^\infty =\alpha(I)^k\: s^\infty= \alpha(I^k)\: s^\infty.
\end{eqnarray}
An argument as in the proof of Corollary~\ref{strengthen} shows the second equation in (\ref{equal1}).

Set $J_k=(I^\sharp)^k:s^\infty$ and $I_k=(I^k)^*R$. Then $S_J(I^\sharp)$ is finitely generated if and only if for some integer $d>0$ one has $(J_k)^d=J_{dk}$ for all $k$, and a corresponding statement holds  for $A(I)$, see for example \cite[Theorem]{HHT} or \cite{R}.

For an $R[[s]]$-module $M$, we set $\overline{M}=M/sM$. Then since $s$ is regular on $R[[s]]/J_{dk}$, the exact sequence
\[
0\to J_{dk}/(J_k)^d\to R[[s]]/(J_k)^d\to R[[s]]/J_{dk}\to 0
\]
induces the exact sequence
\[
0\to \overline{J_{dk}/(J_k)^d}\to R/(I_k)^d\to R/I_{dk}\to 0,
\]
see \cite[Proposition 1.1.4]{BH}.  Therefore,  $\overline{J_{dk}/(J_k)^d}=I_{dk}/(I_k)^d$, and hence Nakayama's lemma implies that $(J_k)^d=J_{dk}$ if and only if $(I_k)^d=I_{dk}$, which is the case if and only if $((I^k)^*)^d=(I^{dk})^*$. This shows that $S_J(I)$ is finitely generated if and only if $A(I)$ is finitely generated. In the same way one shows that $S_J(I)$ is standard graded if and only if $A(I)$ is standard graded.
\end{proof}

\begin{Example}
\label{new}
{\em Let $J=(x^2,xy-sy^3,y^5)\subset K[[x,y,s]]$. Then $J=I^\sharp$ for the ideal $I=(x^2,xy-y^3)$ in Example~\ref{infinite}. The ideal $J$ is a Cohen--Macaulay ideal of codimension 2 and has the relation matrix
\[
\begin{pmatrix}
y & -x+sy^2 & s^2\\
0 & -y^4 & x-sy^2
\end{pmatrix}
\]
Since $A(I)$ is not finitely generated,
Proposition~\ref{comparison} tells us that $\Dirsum_k J^k\:s^\infty$ is not finitely generated as well.

Since $(J,s)$ is $(x,y,s)$-primary it follows that $J^k\:s^\infty=\widetilde{J^k}$.  Computations with CoCoA suggests that this limit exists and  is equal to  $1/3$.
}
\end{Example}

\end{document}